\documentclass[12pt]{amsart}
\usepackage[english]{babel}
\usepackage[margin=1.5in]{geometry}

\usepackage{graphicx,amssymb,latexsym,amsmath,amscd,amsthm,amsfonts}
\usepackage{caption}
\usepackage{subcaption}
\usepackage{tabularx}
\usepackage{verbatim, color}

\usepackage{epstopdf}

\ifpdf
  \DeclareGraphicsExtensions{.eps,.pdf,.png,.jpg}
\else
  \DeclareGraphicsExtensions{.eps}
\fi
\usepackage{verbatim, color} 

\theoremstyle{plain}
\newtheorem{thm}{Theorem}[section]

\newtheorem{prop}[thm]{Proposition}

\newtheorem{lemma}[thm]{Lemma}
\newtheorem{exam}[thm]{Example}
\newtheorem{rem}[thm]{Remark}
\newtheorem{rems}[thm]{Remarks}

\newcommand{\E}{\mathcal{E}}
\newcommand{\G}{\mathcal{G}}

\begin{document}

\title{Stabilizing Graph-dependent Switched Systems}

\author{Nikita Agarwal}
\address{Department of Mathematics, Indian Institute of Science Education and Research Bhopal, Bhopal Bypass Road, Bhauri, Bhopal 462 066, Madhya Pradesh, INDIA\\
nagarwal@iiserb.ac.in}

\begin{abstract}
We give sufficient conditions for stability of a continuous-time linear switched system consisting of finitely many subsystems. The switching between subsystems is governed by an underlying graph. The results are applicable to switched systems having some or all non-Hurwitz subsystems. We also present a slow-fast switching mechanism on subsystems comprising simple loops of underlying graph to ensure stability of the switched system.
\end{abstract}

\maketitle

\section{Introduction} \label{intro} 
We consider a continuous-time switched system which is a piecewise continuous dynamical system consisting of finitely many subsystems. The switching between subsystems is determined by a switching signal which is a piecewise constant function. The signal is represented by the admissible switching from one subsystem to another using the architecture of an underlying directed graph and the times at which these switchings take place. The system can switch from one subsystem to another if there is an edge between the corresponding vertices on the underlying graph. Such systems have been studied in~\cite{NA,IK,KS,O1,KIO,Man}. Switched systems have applications in electrical and power grid systems, where the underlying graph structure varies with time. The networks whose topology changes randomly have been studied in~\cite{AP,B3,B2,PZ,PSBS,PSB}. Synchronization in oscillator networks with varying underlying topology is discussed in~\cite{LCWLT,PJ}. We refer to an editorial by Belykh~\emph{et al.}~\cite{B1} for a review on switched systems as an evolving dynamical system and its potential applications. 

The stability of a switched system not only depends on the properties of subsystems but also on the switching signal. It is known that a switched system with all stable subsystems can be unstable for a particular switching signal. On the other hand, there are several conditions in the literature under which a switched system is stable for arbitrary signals, see Liberzon~\cite{Lib}. Using dwell time and average dwell time approach, sufficient conditions are present in the literature to ensure stability of switched system with all stable subsystems, see~\cite{NA,GC,KS,O1,KIO,Morse1}. Sufficient stability conditions for planar switched systems with all Hurwitz subsystems are discussed in~\cite{BBM}. In~\cite{HXMA}, stability results are presented for the case when all the subsystem matrices commute pairwise. Moreover, for switched systems where some unstable subsystems are present with atleast one stable subsystem, there are sufficient conditions under which the switched system can be stabilized. Stability of switched systems with both stable and unstable subsystems are discussed in~\cite{SU}, using average dwell time approach. For switched positive linear systems having both stable and unstable subsystems, stability results are given in~\cite{ZWXS}. Sufficient conditions in terms of the network topology and also using the concept of flee time from an unstable subsystem and dwell time in a stable system are given in~\cite{NA}. In their paper, using the concept of standard decomposition, a concept of simple loop dwell time is introduced to get a slow-fast switching mechanism. Of course, such systems are not stable under arbitrary signals since for a constant switching signal which keeps the system in the unstable subsystem, the switched system is unstable. \\
Further, it is also known that even when all subsystems are unstable, the switched system can be stable for some switching signal, see~\cite{Lib}. In this case, finding sufficient conditions for stability of the switched system is challenging. Most of the results present in the literature use state-dependent switching~\cite{Feron,LM,Lib,WPD,XC}. There are only a few results with respect to time-dependent switching signals which we will now discuss. \\
In 2018, Ma~\emph{et al.}~\cite{LMF} gave a sufficient condition for stability of a discrete-time switched system, which can be easily verified for linear systems. Xiang and Xiao in~\cite{WW} proposed a sufficient condition for stability of a continuous-time linear switched system using discretized Lyapunov function approach. Their condition demands that the time spent by the system in each subsystem is bounded below and above by fixed quantities. \\
In this paper, we provide a set of new sufficient conditions for stability of the switched system, which are given in Theorem~\ref{thm:main}. Our conditions are in terms of the Jordan decomposition of the subsystem matrices and the underlying graph. As in~\cite{WW}, our sufficient conditions also gives a lower bound and an upper bound on the (dwell) times that the switched system spends in each subsystem. In certain cases, we will see that there is no lower bound on the dwell time. Moreover, we provide conditions necessary for the hypothesis of our Theorem~\ref{thm:main} to be satisfied (see Remark~\ref{rem:hyp} (5) and Proposition~\ref{prop:trace}). \\
Another useful feature of our sufficient conditions, which are in terms of the spectral norm and Jordan decomposition, is that it is easy to check, in contrast to the sufficient conditions given in the existing literature~\cite{WW}, where one needs to solve a large number of matrix inequalities. For planar systems, in particular, our conditions reduce to solving certain inequalities presented in Section~\ref{sec:planar}. The Jordan decomposition technique was also used by the author in~\cite{NA} and Karabacak in~\cite{O1} for situations when all subsystems are stable or atleast one subsystem is stable. The sufficient conditions given in this paper reduce to conditions given in ~\cite{NA,O1} when all subsystems are stable.\\
The paper is organized as follows: in Section~\ref{background}, we give the necessary background material, discuss results in Section~\ref{results}, and specifically focus on planar systems in Section~\ref{sec:planar}. We give numerical examples illustrating our results in Section~\ref{sec:exam}. Finally, in Section~\ref{conclusion}, we summarize our results and discuss future directions.

\section{Background} \label{background} 
In this section, we give preliminaries on directed graphs and describe a switched linear continuous-time system whose switching is given by a finite or an infinite path on a fixed underlying graph. For a $n\times n$ matrix $M=(m_{ij})$, $\Vert M\Vert$ will denote its \textit{spectral norm}, $\rho(M)$ its \textit{spectral radius}, and $s_n(M)\geq 0$ the \textit{smallest singular value} of $M$ which is the square root of the smallest (real) eigenvalue of $M^tM$. The \textit{Frobenius norm} of $M$, denoted by $\Vert M\Vert_F$, is defined as the trace of $M^tM$, which is equal to $\sum_{i=1}^n\sum_{j=1}^n\vert m_{ij}\vert ^2$. Clearly $\Vert M\Vert_F\geq \Vert M\Vert$.

\subsection{Graph-dependent switched system}
A directed graph (or a digraph) consists of a set of vertices and directed edges from one vertex to another. For a graph $\G$ with $k$ vertices, we label them as $v_1,\dots,v_k$. The set of vertices $\{v_1,\dots,v_k\}$ is denoted by $v(\G)$. The edge set, denoted by $\mathcal{E}(\mathcal{G})$, is the collection of all ordered tuples $(i,j)$, where there is an edge from vertex $v_i$ to $v_j$, for $i,j\in\{1,\dots,k\}$. A \emph{path} in the graph $\mathcal{G}$ is a sequence of 
vertices and directed edges such that from each vertex there is an edge to the next vertex in the sequence. The number of edges describing a path $p$ is called the \emph{length of the path}, denoted by $\ell(p)$. If the sequence of vertices in the path $p$ is $v_{i_1},\dots,v_{i_{\ell(p)}}$, we will denote the path as $p=(i_1,\dots,i_{\ell(p)})$. A path whose terminal vertices are the same is called a \emph{loop}. An \emph{acyclic graph} is a graph without any loops. A loop having all distinct vertices is called a \emph{simple loop}. Every loop can be uniquely expressed as a union of simple loops, see Section~\ref{stddec} for standard decomposition algorithm. \\
Let $\G$ be a directed graph with $k$ vertices $\{v_1,\dots,v_k\}$ and no self-loops (an edge from a vertex to itself). Let $\sigma:[0,\infty)\rightarrow \{1,\dots,k\}$ be a piecewise constant right-continuous function with discontinuities $0=t_0<t_1<t_2<\dots$, such that $(\sigma(t_i),\sigma(t_{i+1}))\in\mathcal{E}(\G)$, for all $i\geq 0$. Let $\sigma_i$ denote the value of $\sigma$ in the time interval
$[t_{i-1},t_i)$, for $i\geq 1$. Thus $(\sigma_1,\dots,\sigma_m)$ is a path of length $m$ in $\G$. We call such a signal $\sigma$, a \textit{$\G$-admissible signal}. Each $\G$-admissible signal comprises of the following: switching times $(t_n)_{n\geq 1}$ and a path in $\G$ given by the sequence $(\sigma_n)_{n\geq 1}$. The collection of all $\G$-admissible signals is denoted by $\mathcal{S}_\G$. We now define a sub-class of the collection of switching signals $\mathcal{S}_\G$, which we will use in this article. Label the edges of $\G$ as $e_1,\dots,e_\ell$, where $\ell$ is the number of edges in $\G$. Define
\begin{eqnarray}\label{eq:class}
%\mathcal{S}_\G(\tau_1,\dots,\tau_\ell) &=& \{\sigma\in \mathcal{S}_\G\ \vert \ \text{time spent on edge $e_i$ is $\tau_i$}\},\nonumber \\
\mathcal{S}_\G(I_1,\dots,I_\ell) &=& \{\sigma\in \mathcal{S}_\G\ \vert \ \text{time $\tau_i$ spent on edge $e_i$ satisfies $\tau_i\in I_i$}\},
\end{eqnarray}
where $I_i$ is an open sub-interval of $(0,\infty)$, for each $i=1,\dots,\ell$. 

\noindent Let $A_1,\dots,A_k$ be $n\times n$ matrices with real entries. We call a matrix \textit{stable} (or Hurwitz) if all its eigenvalues have negative real part, and a matrix is called \textit{unstable} if it has at least one eigenvalue with positive real part. A matrix which is not Hurwitz will be called \textit{non-Hurwitz} throughout the paper. 

\noindent For $\sigma\in \mathcal{S}_\mathcal{G}$, consider the switched linear system in $\mathbb{R}^n$ given by 
\begin{equation}\label{main} 
	x^\prime(t) = A_{\sigma(t)}x(t), \ \ t\geq 0.
\end{equation} 
The system~(\ref{main}) is called \textit{a switched system with a $\G$-admissible signal $\sigma\in \mathcal{S}_\G$}. For each $i\geq 1$, the linear system $x^\prime(t) = A_{\sigma_i}x(t)$, $t\in [t_{i-1},t_i)$, 
is called a subsystem of~(\ref{main}). This subsystem is known as \textit{stable} (respectively, unstable, non-Hurwitz) if $A_{\sigma_i}$ is a stable (respectively, unstable, non-Hurwitz) matrix. The matrices $A_1,\dots,A_k$ are called \textit{subsystem matrices of the switched system}.  We will now discuss the stability notions for switched systems.\\
A graph-dependent switched system~(\ref{main}) with $\sigma\in\mathcal{S}_{\G}$ is \textit{globally exponentially stable} if there exist positive constants $\alpha$ and $\beta$ such that for all initial conditions $x(0)\in \mathbb{R}^n$, $\Vert x(t)\Vert \leq \alpha e^{-\beta t}\Vert x(0)\Vert$. In this paper, we will discuss sufficient conditions which will guarantee global exponential stability of the switched system. Since we will only discuss global exponential stability, we will just call it \textit{stability} for convenience. We will need the following lemmas in this paper.

\begin{lemma}\label{lemma_1}
If $A$ is a $n\times n$ invertible matrix then $\Vert A^{-1}\Vert = \dfrac{1}{s_n(A)}$.
\end{lemma}
\begin{proof}
Since $s_n(A)= \inf_{x\ne 0} \dfrac{\Vert Ax\Vert }{\Vert x\Vert}$, we get 
\[
\dfrac{1}{s_n(A)}= \sup_{x\ne 0} \dfrac{\Vert x\Vert}{\Vert Ax\Vert}= \sup_{y=Ax\ne 0} \dfrac{\Vert A^{-1}y\Vert}{\Vert y\Vert} = \Vert A^{-1}\Vert.
\]
\end{proof}

\begin{lemma}\label{lemma_2}
If $A$ and $B$ are invertible matrices of size $n$ with $\Vert A\Vert\geq 1$, $\Vert B\Vert\geq 1$, and $\Vert AB\Vert <1$, then $s_n(A)<1$ and $s_n(B)<1$.
\end{lemma}
\begin{proof}
The result follows from the inequality
\[
1>\Vert AB\Vert\geq \max\{s_n(A)\Vert B\Vert, s_n(B)\Vert A\Vert\}\geq \max\{s_n(A), s_n(B)\}.
\]
\end{proof}

\subsection{Standard Decomposition Algorithm}  \label{stddec}
Let $\G$ be a directed graph with vertex set $\{v_1,\dots,v_k\}$. Consider a signal $\sigma\in \mathcal{S}(\G)$, with associated switching times $(t_n)_{n\geq 1}$ and an infinite path $(\sigma_n)_{n\geq 1}$ in $\G$, with edges $e_n=(v_{\sigma_n},v_{\sigma_{n+1}})$, $n\geq 1$.  In~\cite{NA}, a standard decomposition algorithm of paths $(\sigma_1,\dots,\sigma_n)$, $n\geq 1$, into simple loops and an indecomposable path was introduced which we describe now.\\
\noindent \textbf{Step 1}: Let $p_0=(\sigma_1,\sigma_2,\dots\sigma_n)$ be the path with edges $e_1, e_2,\dots, e_{n-1}$, and let $i(p_0)$ be the set consisting of subscripts $j$ of all $e_j$ that comprise $p_0$. Let $r_2\in i(p_0)$ be the minimum index such that $\sigma_{r_2}=\sigma_{j+1}$ for some $j<r_2$ in the index set $i(p_0)$ (that is, the initial vertex of $e_{r_2}$ is the terminal vertex of $e_j$). Let $r_1\in i(p_0)$ be such that $r_1<r_2$ and $\sigma_{r_1+1}=\sigma_{r_2}$. If such a pair does not exist, then the path $p_0$ is indecomposable and the algorithm stops. Otherwise, we proceed to Step 2. It is easy to see that the subpath $p^0=(\sigma_{r_1+1},\dots,\sigma_{r_2})$ of $p_0$ with edges $e_{r_1+1},\dots,e_{r_2-1}$ is a simple loop in $\G$.\\\\
\noindent \textbf{Step 2}: Let $p_1=(\sigma_1,\dots,\sigma_{r_1+1},\sigma_{r_2+1},\dots,\sigma_n)$ be the path obtained by deleting the edges of $p^0$ from $p_0$. If $p_1$ is indecomposable, the algorithm stops, otherwise repeat Step 1 by replacing $p_0$ by $p_1$. \\~\\
Using this algorithm, $\sigma^{(n)}$ can be decomposed into simple loops and an indecomposable path. Such a decomposition is called the \textit{standard decomposition}. Note that the steps of this decomposition can be used to express any loop in $\G$ into a union of simple loops. 

\begin{exam}
Let $p_0=(1,2,3,2,3,1,2)$ be a path in $\G$ (see Figure~\ref{fig:stddec}). It can be checked that $p^0=(\sigma_2,\sigma_3,\sigma_4)=(2,3,2)$, thus $p_1=(\sigma_1,\sigma_2,\sigma_5,\sigma_6,\sigma_7)=(1,2,3,1,2)$. Then $p^1=(\sigma_1,\sigma_2,\sigma_5,\sigma_6)=(1,2,3,1)$, thus $p_2=(\sigma_1,\sigma_7)=(1,2)$ which is an indecomposable path. Thus $p_0$ is a union of simple loops $p^0$, $p^1$ and an indecomposable path $p_2$.

\begin{figure}[h!]
\centering
\includegraphics[width=6cm,height=5cm]{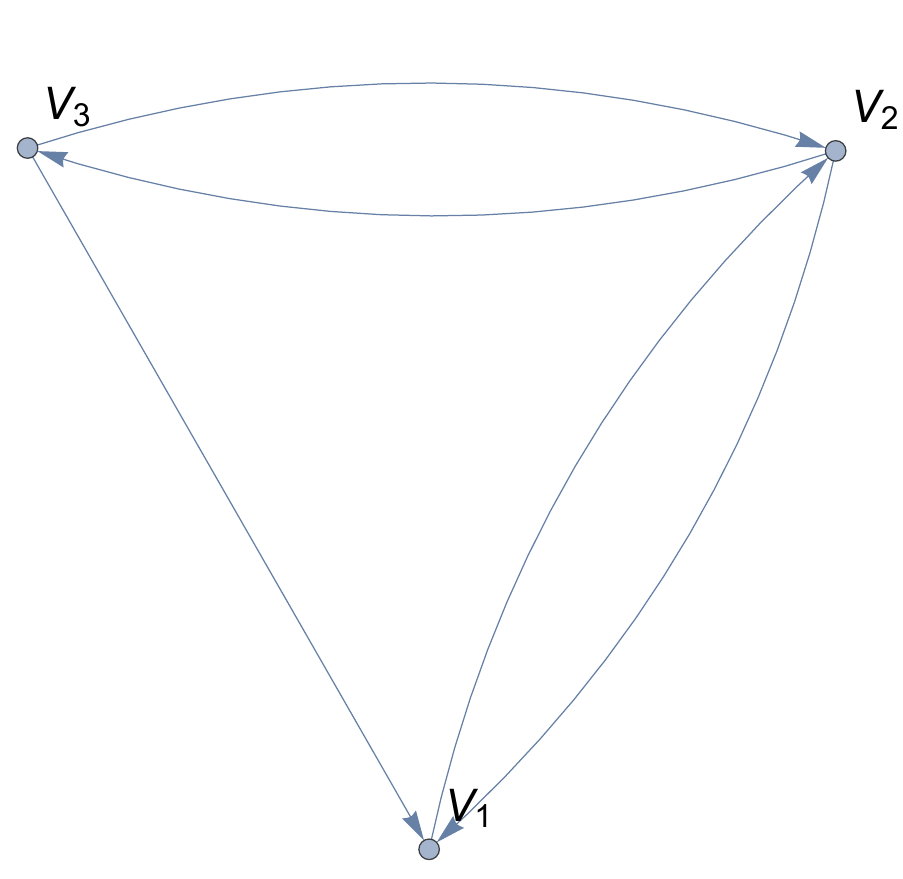}
\caption{The graph $\G$.}
\label{fig:stddec}
\end{figure}

%\begin{wrapfigure}{R}{0.3\textwidth}
%\centering
%\includegraphics[width=0.25\textwidth]{example-stddec.pdf}
%\caption{\label{fig:stddec}The graph $\G$.}
%\end{wrapfigure}

\end{exam}

\section{Results} \label{results}
Consider the switched system~(\ref{main}) and let $\lambda_i$ be the real part of the eigenvalue of $A_i$ with maximum real part, for each $i=1,\dots,k$. We assume the following hypotheses (H), see Remark~\ref{rem_assump}.\\

\noindent $(\mathbf{H})$ The switching signal $\sigma\in \mathcal{S}(\G)$ has infinitely many discontinuities $(t_n)_{n\geq 0}$ and $t_n\rightarrow \infty$ as $n\rightarrow \infty$. 
%
%\noindent $(\mathbf{H2})$ For each $j\in\{1,\dots,k\}$, $A_j$ is a diagonalizable matrix over $\mathbb{C}$. 

\begin{rem}\label{rem_assump}
If (H) is not satisfied then there exists $T\geq 0$ such that $\sigma(t)=j$, for all $t\geq T$, for some $j\in\{1,\dots,k\}$. Hence the switched system is stable if and only if the switched system with constant switching signal with value $j$ is stable. Moreover, (H) implies that the graph $\G$ is not acyclic, that is, it has atleast one loop.
%(2) If $A$ is a diagonalizable matrix (over $\mathbb{C}$) with eigenvalue matrix $D$, then $\Vert e^{Ds}\Vert \leq e^{\mu s}$, where $\mu$ is the real part of the eigenvalue with the maximum real part. If $A$ is not diagonalizable, then for each $\lambda^*>\mu$, there exists $\beta>0$ such that $\Vert e^{Ds}\Vert \leq \beta e^{\lambda^* s}$. 
\end{rem}

\noindent For the remainder of the paper, we will consider stability issue of the switched system~(\ref{main}) with $\G$, $\sigma$, and $A_1,\dots,A_k$ satisfying (H). By (H), it is clear that zeno behavior does not occur in the switched systems under consideration. We now state and prove our main result giving sufficient conditions for stability of switched system.

\begin{thm}\label{thm:main}
If there exist invertible matrices $P_1,\dots,P_k$ such that $P_iJ_iP_i^{-1}$ is a Jordan decomposition of $A_i$, for $i=1,\dots,k$, and for each $(r,s)\in \mathcal{E}(\G)$, there exists $\eta_{(r,s)}>0$ such that
\begin{eqnarray}\label{ineq:hyp}
\Vert P_{(r,s)}e^{\eta_{(r,s)} J_r}\Vert &<& 1,
\end{eqnarray}
where $P_{(r,s)}=P_s^{-1}P_r$, then the switched system~(\ref{main}) is stable for every switching signal $\sigma\in \mathcal{S}_\G(I_1,\dots,I_\ell)$, where $I_i$ is some open interval in $(0,\infty)$ containing $\eta_{(r,s)}$ with $e_i=(r,s)$, $i=1,\dots,\ell$. 
\end{thm}

\begin{rems}\label{rem:hyp}
1) It should be noted that if $\Vert P_{(r,s)}\Vert\geq 1$, then the left end point of $I_{i}$ is strictly positive, where $e_i=(r,s)$. \\
2) Note that the inequalities~(\ref{ineq:hyp}) imply invertibility of $I-P_{(r,s)}e^{\eta_{(r,s)} J_r}$, for each $(r,s)\in\mathcal{E}(\G)$.\\
3) Since $\G$ has a closed loop by (H) and Remark~\ref{rem_assump}(1), there exist atleast one $(i,j)\in\mathcal{E}(\G)$ such that $\Vert P_{(i,j)}\Vert\geq 1$: If $(i_1,\dots,i_p,i_1)$ is a loop in $\G$, then since
\begin{eqnarray*}
I &=& (P_{i_1}^{-1}P_{i_p})(P_{i_p}^{-1}P_{i_{p-1}})\dots (P_{i_3}^{-1}P_{i_2})(P_{i_2}^{-1}P_{i_1})\\
&=& P_{(i_p,i_1)}P_{(i_{p-1},i_p)}\dots P_{(i_2,i_3)}P_{(i_1,i_2)},
\end{eqnarray*}
we get $1\leq \Vert P_{(i_p,i_1)}\Vert \Vert P_{(i_{p-1},i_p)}\Vert \dots \Vert P_{(i_2,i_3)}\Vert\Vert P_{(i_1,i_2)}\Vert$.\\
Let $P_1,\dots,P_n$ satisfy the hypothesis in the statement of Theorem~\ref{thm:main}, and let 
\begin{eqnarray}\label{eq:e1e2}
\mathcal{E}_1(\G)&=&\{(i,j)\in\mathcal{E}(\G) \ \vert\ \Vert P_{j}^{-1}P_i\Vert\geq 1\}\ne \emptyset, \text{ and} \nonumber \\
\mathcal{E}_2(\G)&=&\mathcal{E}(\G)\setminus \mathcal{E}_1(\G). 
\end{eqnarray}
4) For all $e_i=(r,s)\in\mathcal{E}_2(\G)$, since $\Vert P_{s}^{-1}P_r\Vert<1$, there exists $\eta_{(r,s)}>0$ such that $\Vert P_{s}^{-1}P_re^{J_r\eta_{(r,s)}}\Vert<1$. Hence the hypothesis of Theorem~\ref{thm:main} is satisfied. \\
In particular, when $A_r$ is diagonalizable over $\mathbb{C}$, we have $\Vert P_{(r,s)}e^{\eta_{(r,s)} J_r}\Vert \leq \Vert P_{(r,s)}\Vert e^{\eta_{(r,s)} \lambda_r} <1$ provided 
\[
0<\eta_{(r,s)}<-\dfrac{\ln\Vert P_{(r,s)}\Vert}{\lambda_r},
\]
when $\lambda_r>0$. In this case, we can take $I_{i}=\left(0,-\dfrac{\ln\Vert P_{(r,s)}\Vert}{\lambda_r}\right)$. For $e_i=(r,s)\in\mathcal{E}_2(\G)$, if $\lambda_r\leq 0$, we can take $I_{i}=\left(0,\infty\right)$. \\
Thus for a given set of matrices $P_1,\dots,P_k$, it is enough to check hypothesis of Theorem~\ref{thm:main} for $(r,s)\in\mathcal{E}_1(\G)$.  \\
5) Let $P_1,\dots,P_n$ satisfy the hypothesis in the statement of Theorem~\ref{thm:main}, then by Lemma~\ref{lemma_2}, for $(r,s)\in\mathcal{E}_1(\G)$, $s_n(e^{J_r})<1$ since $\Vert P_{(r,s)}\Vert\geq 1$. Moreover, if $A_r$ is diagonalizable over $\mathbb{C}$, $s_n(e^{J_r})<1$ holds if and only if $A_r$ has an eigenvalue with negative real part. Note that this is not true for non-diagonalizable case: the smallest singular value $s_2(e^{J})$ of the non-diagonalizable matrix $A=J=\begin{pmatrix}
a&1\\0&a
\end{pmatrix}$ is less than one for all values of $a<0.48$. Hence, for the hypothesis in the statement of Theorem~\ref{thm:main} to be satisfied, for $(r,s)\in\mathcal{E}_1(\G)$ with $A_r$ diagonalizable over $\mathbb{C}$, $A_r$ must have an eigenvalue to the left of the imaginary axis. \\
%3) $\Vert P_{(r,s)}e^{\eta_{(r,s)}J_r}\Vert = \Vert P_s^{-1} e^{\eta_{(r,s)}A_r} P_r\Vert$.
6) If $A_r$ was stable matrix, that is, $\lambda_r<0$, then for each $0>\lambda_r^*>\lambda_r$, there exists $\beta>0$ such that $\Vert P_{(r,s)}e^{t J_r}\Vert\leq \beta\Vert P_{(r,s)}\Vert e^{t\lambda_r^*}$ will be less than 1 for all $t>0$ satisfying 
\[
t > -\dfrac{\ln \left(\beta \Vert P_{(r,s)}\Vert\right)}{\lambda_r^*},
\]
for any choice of $P_r, P_s$. See Agarwal~\cite{NA}, Karabacak~\cite{O1} and references therein for related bounds on dwell time in case of all stable subsystems. Further we refer to~\cite{NA} for stability of switched system having atleast one stable subsystem and when the subgraph of $\G$ corresponding to unstable subsystems is acyclic. 
\end{rems}

\begin{proof} (Proof of Theorem~\ref{thm:main})
If for each $(r,s)\in \mathcal{E}(\G)$, there exists $\eta_{(r,s)}>0$ such that
\[
\Vert P_{(r,s)}e^{\eta_{(r,s)} J_r}\Vert <1,
\]
then for all $i=1,\dots,\ell$, there exist a bounded interval $I_{i}\subseteq(0,\infty)$ containing $\eta_{(r,s)}$ such that 
\[
\Vert P_{(r,s)}e^{\eta J_r}\Vert <1,
\]
for all $\eta\in I_{i}$, where $e_i=(r,s)$. We show that the switched system~(\ref{main}) is stable for all $\sigma\in \mathcal{S}_\G(I_1,\dots,I_\ell)$.\\
For $t\in [t_{n-1},t_n)$, the solution of the switched system~(\ref{main}) with initial condition $x(0)$ is given by $x(t)=
e^{A_{\sigma_n}(t-t_{n-1})}e^{A_{\sigma_{n-1}}(t_{n-1}-t_{n-2})}\dots e^{A_{\sigma_{1}}t_{1}} x(0)$. Using Jordan decomposition $A_i=P_iJ_iP_i^{-1}$, we get
\begin{eqnarray} \label{eq:1}
\Vert x(t)\Vert &=& \Vert e^{A_{\sigma_n}(t-t_{n-1})}e^{A_{\sigma_{n-1}}(t_{n-1}-t_{n-2})}\dots e^{A_{\sigma_{1}}t_{1}} x(0)\Vert   \nonumber \\ 
&\leq &  \Vert P_{\sigma_n}e^{J_{\sigma_n}(t-t_{n-1})}\Vert \Vert P_{\sigma_1}^{-1}\Vert  \left(\prod_{j=1}^{n-1}\Vert
P_{(\sigma_{j},\sigma_{j+1})} e^{J_{\sigma_j}(t_j-t_{j-1})}\Vert\right)\Vert x(0)\Vert \nonumber \\ 
&\leq &  C a_n^\sigma \Vert x(0)\Vert, 
\end{eqnarray} 
where the constant $C>0$ is given by
\[
C = \sup\{\Vert P_se^{t J_s}\Vert \Vert P_r^{-1}\Vert \ \vert \ (t,r,s)\in \Lambda\},
\]
with $\Lambda$ is the collection of all triples $(t,r,s)$ with $r,s\in\{1,\dots,k\}$ such that there is a path (of any length) from $v_r$ to $v_s$ and $t\in I_{i}$, for edges $e_i$ originating at the vertex $v_s$, $i=1,\dots,\ell$. The constant $C$ is independent of $\sigma$ and $n$, for all $n$, and
\begin{eqnarray*}
a_n^\sigma &=& \prod_{j=1}^{n-1}\Vert
P_{(\sigma_{j},\sigma_{j+1})} e^{J_{\sigma_j}(t_j-t_{j-1})}\Vert.
\end{eqnarray*}
Each term in the product is less than $K<1$, where 
\[
K=\sup\{\Vert P_{(r,s)}e^{t J_r}\Vert \ \vert \ t\in I_{i},\ e_i=(r,s)\in\mathcal{E}(\G),\ i=1,\dots,\ell\}.
\]
Hence $a_n^\sigma\rightarrow 0$ as $n\rightarrow \infty$ (at an exponential rate). Thus the switched system~(\ref{main}) is stable for every switching signal $\sigma\in \mathcal{S}_\G(I_1,\dots,I_\ell)$.
\end{proof}

\begin{prop}\label{prop-unit_norm}
If $P_i$, $Q_i$, $J_i$, $K_i$, $i=1,\dots,k$ are matrices such that $A_i=P_iJ_iP_i^{-1}$ and $A_i=Q_iK_iQ_i^{-1}$ are Jordan decompositions of $A_i$ with $P_i$ and $Q_i$ having all columns with unit norm, then the following are equivalent:\\
1)  For each $(r,s)\in \mathcal{E}(\G)$, there exists $\eta_{(r,s)}>0$ such that 
\[
\Vert P_{(r,s)}e^{\eta_{(r,s)} J_r}\Vert <1,
\]
where $P_{(r,s)}=P_s^{-1}P_r$.\\
2) For each $(r,s)\in \mathcal{E}(\G)$, there exists $\zeta_{(r,s)}>0$ such that 
\[
\Vert Q_{(r,s)}e^{\zeta_{(r,s)} K_r}\Vert <1,
\] 
where $Q_{(r,s)}=Q_s^{-1}Q_r$.
\end{prop}
\begin{proof}
For each $i=1,\dots,k$, since $J_i=R_iK_iR_i^t$, for some rotation matrix $R$, 
\[
P_iJ_iP_i^{-1}=P_iR_iK_iR_i^{-1}P_i^{-1}=Q_iK_iQ_i^{-1}.
\]
Hence $Q_i=P_iR_iU_i$ for some unitary matrix $U_i$. Thus for $(r,s)\in\E(\G)$,
\[
\Vert P_{(r,s)}e^{\eta_{(r,s)} J_r}\Vert = \Vert U_s^{-1}R_s^{-1}Q_{(r,s)}R_rU_re^{\eta_{(r,s)} J_r}\Vert = \Vert Q_{(r,s)}e^{\eta_{(r,s)} J_r}\Vert.
\]
Take $\zeta_{(r,s)}=\eta_{(r,s)}$.
\end{proof}

\begin{rem}
In view of Proposition~\ref{prop-unit_norm}, if the eigenvector matrices $P_1,\dots,P_k$ satisfying the hypothesis of Theorem~\ref{thm:main} have unit norm columns, then the hypothesis of the theorem are satisfied for any choice of eigenvector matrices with unit norm columns. In Example~\ref{exam:1}, we will see that the hypothesis of Theorem~\ref{thm:main} may not be satisfied by eigenvector matrices with unit norm, but may be satisfied with appropriate scaling of eigenvectors. \\
If the columns of the invertible matrices $P_1,\dots,P_k$ have unit norm, then by Proposition~\ref{prop-unit_norm}, we can fix a choice of these matrices and corresponding $J_1,\dots,J_k$ such that   $A_i=P_iJ_iP_i^{-1}$ is a Jordan decomposition of $A_i$, $i=1,\dots,k$, and then the hypothesis of Theorem~\ref{thm:main} is equivalent to finding invertible diagonal matrices $D_1,\dots,D_k$ such that 
\begin{eqnarray}\label{eq:5}
\Vert D_s^{-1}P_{(r,s)}D_r e^{J_r \eta_{(r,s)}}\Vert<1,
\end{eqnarray}
for some $\eta_{(r,s)}>0$, $(r,s)\in\mathcal{E}(\G)$.
\end{rem}

\begin{exam}\label{exam:1}
In this example, the hypothesis of Theorem~\ref{thm:main} is not satisfied if we take $P_i$ having unit norm columns, whereas inequalities~(\ref{eq:5}) are satisfied for some choice of $D_1,\dots,D_k$.\\
Consider a switched system on a unidirectional ring with two vertices as the underlying graph and with planar subsystems with  $A_1=\text{diag}(-1,1)$ and $A_1=\text{diag}(1,-2)$.
Both $A_1$ and $A_2$ are unstable. If we insist on columns of $P_i$ having unit norm, then $P_1=P_2=I$ with $J_1=A_1$ and $J_2=A_2$ (we can make this choice in view of Proposition~\ref{prop-unit_norm}).
Clearly with these choices of $P_1$ and $P_2$, the hypothesis in Theorem~\ref{thm:main} is not satisfied. But if we take 
\[
D_1=\text{diag}(e^2,e^{-3}), \ D_2=I,
\] 
inequalities~(\ref{eq:5}) are satisfied, for all $\eta_{(1,2)}\in(2,3)$ and $\eta_{(2,1)}\in(1.5,2)$.\\
This example is special because if $T_1$ is the time spent in subsystem $A_1$ and $T_2$ is the time spent in subsystem $A_2$, then the switched system is stable if $T_2<T_1<2T_2$ since then $\Vert e^{A_2T_2}e^{A_1T_1}\Vert = \Vert e^{A_2T_2+A_1T_1}\Vert <1$. This example can be generalized to a unidirectional graph with $k$ vertices and pairwise commuting subsystem matrices $A_1,\dots,A_k$. If there exist $T_1,\dots,T_k>0$ such that $\Vert e^{T_1A_1+\dots+T_kA_k}\Vert<1$, then the switched system is stable for some switching signal. In particular, if a convex combination of $A_1,\dots,A_k$ is Hurwitz, then the switched system is stabilized. This condition of existence of convex Hurwitz combination appears in quadratic stability of switched system via state dependent switching, see Liberzon~\cite{Lib}. Further if each $A_1,\dots,A_k$ is a diagonal matrix, then the switched system with unidirectional ring as the underlying graph is stable if and only if a convex combination of $A_1,\dots,A_k$ is Hurwitz.
\end{exam}

\subsection{Few Estimates when $\mathcal{E}_2(\G)\ne\emptyset$} We recall subsets of the edge set $\mathcal{E}(\G)$ defined in equation~(\ref{eq:e1e2}). Let the hypotheses of Theorem~\ref{thm:main} be satisfied for a choice of $P_1,\dots,P_k$ and intervals $I_{i}$, for each $e_i=(r,s)\in\mathcal{E}_2(\G)$. Let $s_1\dots,s_p$ be simple loops in $\G$ (recall Remark~\ref{rem_assump}(1)). For each $j=1,\dots,p$, let $\eta^j_{(r,s)}>0$ be the time that the system spends in each subsystem $r$ before switching to subsystem $s$ with $(r,s)\in \mathcal{E}(s_j)$. We assume that $\eta^j_{(r,s)}\in I_{i}$, for all $e_i=(r,s)\in \mathcal{E}(s_j)\cap\mathcal{E}_2(\G)$. For each $e_i=(r,s)\in \mathcal{E}(s_j)\cap\mathcal{E}_2(\G)$, let
\[
K^j_{(r,s)}=\sup\{\Vert P_{(r,s)}e^{t J_r}\Vert \ \vert \ t\in I_{i}\}<1. 
\] 

\subsubsection{Bound on the total time spent on edges in $\mathcal{E}_2(\G)$ on each simple loop} For $j=1,\dots,p$, let
\begin{eqnarray}
M_j &=& \sum_{(r,s)\in\mathcal{E}(s_j)\cap\mathcal{E}_2(\G)}\ln \Vert
P_{(r,s)}\Vert <0,\\
N_j &=& \sum_{(r,s)\in\mathcal{E}(s_j)\cap\mathcal{E}_1(\G)}\ln K^j_{(r,s)} < 0,\\
\lambda^j &=& \max_{(r,s)\in\mathcal{E}(s_j)\cap\mathcal{E}_2(\G)}\lambda_r.
\end{eqnarray}
Recall proof of Theorem~\ref{thm:main} and with notation as before, since every finite path in $\G$ can be decomposed into simple loops and a path of length at most $k-1$ in standard decomposition (described in Section~\ref{stddec}), we get
\begin{eqnarray*}
\ln a_n^\sigma &=& b_n^\sigma + \sum_{j=1}^p n_j^\sigma \left[ \sum_{(r,s)\in\mathcal{E}(s_j)}\ln\Vert P_{(r,s)}e^{J_r\eta^j_{(r,s)}}\Vert \right] \\
&\leq & b_n^\sigma + \sum_{j=1}^p n_j^\sigma \left[ M_j + \lambda^j\eta^j + N_j\right] 
\end{eqnarray*}
where $\eta^j = \sum_{(r,s)\in\mathcal{E}(s_j)\cap\mathcal{E}_2(\G)}\eta^j_{(r,s)}$. Since $(r,s)\in \mathcal{E}_2(\G)$, $\Vert P_{(r,s)}\Vert < 1$, there is no lower bound on $\eta^j_{(r,s)}>0$ (see also Remark~\ref{rem:hyp}(4)). Hence there is no lower bound on $\eta^j>0$. The term $b_n^\sigma$ corresponds to the indecomposable path in the standard decomposition. \\
As $n\rightarrow\infty$, for some $j=1,\dots,p$, the number $n_j^\sigma$ of simple loops tends to $\infty$. Moreover $b_n^\sigma$ is finite. Hence $\ln a_n^\sigma\rightarrow -\infty$ when for all $j=1,\dots,p$, $M_j + \lambda^j\eta^j + N_j<0$, which is true if $\lambda^j\leq 0$. If $\lambda^j>0$, $M_j + \lambda^j\eta^j + N_j<0$ provided
\[
0<\eta^j< -\dfrac{M_j+N_j}{\lambda^j}.
\]
Thus we have an upper bound on the total time spent on edges $(r,s)\in \mathcal{E}_2(\G)$ that lie on the simple loop $s_j$ in the standard decomposition. This gives a fast slow mechanism on these edges of the loop. \\
The bound described in this section is only applicable when $\mathcal{E}(s_j)\cap \mathcal{E}_2(\G)\ne \emptyset$. 
%For instance if the underlying graph $\G$ is a unidirectional ring with two vertices, then using Lemma~\ref{lemma_1} and Remark~\ref{rem:hyp}(4), $\mathcal{E}_2(\G)=\emptyset$.

\subsubsection{Bound on the maximum time spent on edges in $\mathcal{E}_2(\G)$ on each simple loop}  
For $j=1,\dots,p$, let
\begin{eqnarray}
\gamma^j &=& \sum_{(r,s)\in\mathcal{E}(s_j)\cap\mathcal{E}_2(\G)}\lambda_r.
\end{eqnarray}
Recall proof of Theorem~\ref{thm:main} and with notation as before, since every path in $\G$ can be decomposed into simple loops and a path of length at most $k-1$ in standard decomposition (Section~\ref{stddec}), we get
\begin{eqnarray*}
\ln a_n^\sigma &=& b_n^\sigma + \sum_{j=1}^p n_j^\sigma \left[ \sum_{(r,s)\in\mathcal{E}(s_j)}\ln\Vert P_{(r,s)}e^{J_r\eta^j_{(r,s)}}\Vert \right] \\
&\leq & b_n^\sigma + \sum_{j=1}^p n_j^\sigma \left[ M_j + \gamma^j\zeta^j + N_j\right] 
\end{eqnarray*}
where $\zeta^j>0$ is the maximum time spent on each edge $(r,s)\in \mathcal{E}(s_j)\cap\mathcal{E}_2(\G)$.\\
As $n\rightarrow\infty$, for some $i=j,\dots,p$, the number $n_j^\sigma$ of simple loops tends to $\infty$. Moreover $b_n^\sigma$ is finite. Hence $\ln a_n^\sigma\rightarrow -\infty$ when for all $j=1,\dots,p$, $M_j + \gamma^j\zeta^j + N_j<0$, which is true if $\gamma^j\leq 0$. If $\gamma^j>0$, $M_j + \gamma^j\zeta^j + N_j<0$ provided
\[
0<\zeta^j< -\dfrac{M_j+N_j}{\gamma^j}.
\]
Thus we have an upper bound on the time spent on edges $(r,s)\in \mathcal{E}_2(\G)$ and the simple loop $s_j$. 

\subsection{Planar systems}\label{sec:planar}
In this section, we will focus on switched systems in $\mathbb{R}^2$. A matrix $A$ is called \emph{Schur stable} if $\rho(A)<1$. As an aside, Schur stability of a matrix $A$ is equivalent to the following: for each symmetric positive definite matrix $Q$, there exists a unique positive definite matrix $P$ such that $P-A^tPA-Q=0$, see~\cite{HJ}. For a matrix $A$ of size two, Schur stability of $A$ is equivalent to the following two conditions: $\vert \text{trace}(A)\vert <1+\text{det}(A)$ and $\vert \text{det}(A)\vert<1$, we refer to~\cite{KB}. Moreover, for a real matrix $A$, $\Vert A\Vert<1$ if and only if $A^tA$ is Schur stable.

\begin{exam}\label{exam:2}
Let $\mathcal{G}$ be a unidirectional ring with two vertices. Suppose both $A_1$ and $A_2$ are non-Hurwitz matrices of size two which are diagonalizable over $\mathbb{C}$. Every $\G$-admissible switching signal $\sigma$ switches between these two subsystems. Let $A_i=P_iJ_iP_i^{-1}$ be the Jordan decomposition of $A_i$, $i=1,2$. By Remark~\ref{rem:hyp}(3), without loss of generality, we can assume that $\Vert P_{(1,2)}\Vert\geq 1$. Further for the hypothesis $\Vert P_{(1,2)}e^{J_1t_0}\Vert<1$ of Theorem~\ref{thm:main} to be satisfied for some $t_0>0$, using Lemmas~\ref{lemma_1},~\ref{lemma_2} and $\Vert e^{J_1t_0}\Vert\geq 1$, we have $1>\Vert P_{(1,2)}e^{J_1t_0}\Vert\geq \Vert P_{(1,2)}\Vert \Vert e^{J_1t_0}\Vert\geq s_2(P_{1,2})=1/\Vert P_{(2,1)}\Vert$, hence $\Vert P_{(2,1)}\Vert>1$. Also, $\Vert P_{(2,1)} e^{J_2 s_0}\Vert \geq s_2(e^{J_2 s_0})$ and $\Vert P_{(1,2)} e^{J_1 s_0}\Vert \geq s_2(e^{J_1 s_0})$, by Lemma~\ref{lemma_2}. \\
Let us analyze various possibilities for the eigenvalues of $A_1$ and $A_2$. If $A_1$ has complex conjugate pair of eigenvalues $\lambda_1\pm i\mu_1$ with $\lambda_1\geq 0$ (since it is non-Hurwitz), then $\Vert P_{(1,2)} e^{J_1 s_0}\Vert \geq \sigma_2(e^{J_1 s_0})=e^{\lambda_1s_0}\geq 1$. Similarly for $A_2$. Hence for the hypothesis of Theorem~\ref{thm:main} to be satisfied, both $A_1$ and $A_2$ have a pair of real eigenvalues, one non-negative (since they are non-Hurwitz) and other negative (using Remark~\ref{rem:hyp}(5)). \\
Let $J_1=\text{diag}(-\alpha_1,\alpha_2)$ and $J_2=\text{diag}(-\beta_1,\beta_2)$, with $\alpha_1, \beta_1>0$, $\alpha_2, \beta_2\geq 0$. \\
The conditions in Theorem~\ref{thm:main} are:
\begin{eqnarray}\label{eq:4}
\Vert P_1^{-1}P_2 e^{J_2 s_0}\Vert<1,\text{ and } \Vert P_2^{-1}P_1 e^{J_1 t_0}\Vert<1,
\end{eqnarray}
for some $t_0, s_0>0$.
It should be noted that if inequalities~(\ref{eq:4}) are satisfied then
\begin{eqnarray*}
\Vert P_1^{-1}P_2 e^{J_2 s_0}P_2^{-1}P_1 e^{J_1 t_0}\Vert &=&\Vert P_1^{-1}e^{A_2 s}e^{A_1 t_0}P_1\Vert <1, \text{ and} \nonumber\\
\Vert P_2^{-1}P_1 e^{J_1 t_0}P_1^{-1}P_2 e^{J_2 s_0}\Vert &=&\Vert P_2^{-1}e^{A_1 t_0}e^{A_2 s_0}P_2\Vert <1.
\end{eqnarray*}
Further observe that $P_1=PD_1$ and $P_2=QD_2$, where $P$ and $Q$ are fixed matrices with all columns having unit norm and $D_1, D_2$ are diagonal matrices with all diagonal entries non-zero. Let $Q^{-1}P=(a_{ij})$. Then $P_2^{-1}P_1 = D_2^{-1}Q^{-1}PD_1$. If $D_1=\text{diag}(p,q)$ and $D_2=\text{diag}(r,s)$, then $P_2^{-1}P_1=\begin{pmatrix}
a_{11}p/r & a_{12}q/r\\
a_{21}p/s & a_{22}q/s
\end{pmatrix}$. The inequalities~(\ref{eq:4}) are satisfied if and only if all of the following conditions hold true: 
\begin{eqnarray}\label{eq:ineq}
T_1<1+D_1,\ D_1<1,\ T_2<1+D_2,\ \text{and } D_2<1,
\end{eqnarray} 
where  
\begin{eqnarray}\label{eq:TD}
T_1&=&e^{-2\alpha_1t_0}p^2\left(\dfrac{a_{11}^2}{r^2}+\dfrac{a_{21}^2}{s^2}\right)+e^{2\alpha_2t_0}q^2\left(\dfrac{a_{12}^2}{r^2}+\dfrac{a_{22}^2}{s^2}\right),\\
D_1&=&e^{2(\alpha_2-\alpha_1)t_0}\left(\dfrac{pq}{rs}(a_{11}a_{22}-a_{12}a_{21})\right)^2,\nonumber\\
T_2&=&\dfrac{e^{-2\beta_1s_0}\dfrac{1}{s^2}\left((a_{21}p)^2+(a_{22}q)^2\right)+e^{2\beta_2s_0}\dfrac{1}{r^2}\left((a_{11}p)^2+(a_{12}q)^2\right)}{\left(\dfrac{pq}{rs}(a_{11}a_{22}-a_{12}a_{21})\right)^2},\nonumber\\
D_2&=&\dfrac{e^{2(\beta_2-\beta_1)s_0}}{\left(\dfrac{pq}{rs}(a_{11}a_{22}-a_{12}a_{21})\right)^2}\nonumber.
\end{eqnarray}
Thus hypotheses of Theorem~\ref{thm:main} are equivalent to solving four inequalities $T_1<1+D_1$, $D_1<1$, $T_2<1+D_2$, and $D_2<1$ in six variables: positive $t_0,s_0$ and non-zero $p,q,r,s$. Further for planar switched system~(\ref{main}) with underlying graph $\G$ having $\ell$ edges, we need to solve $2\ell$ inequalities in $2k+\ell$ variables.\\
Weaker sufficient conditions can be obtained using Frobenius norm. Since the Frobenius norm $\Vert.\Vert_F$ is greater than the spectral norm, inequalities~(\ref{eq:4}) are satisfied if 
\begin{eqnarray}\label{eq:55}
\Vert P_1^{-1}P_2 e^{J_2 s_0}\Vert_F=T_2<1,\text{ and } \Vert P_2^{-1}P_1 e^{J_1 t_0}\Vert_F=T_1<1.
\end{eqnarray}
Note that this is possible only if 
\begin{eqnarray*}
e^{2\alpha_2t_0}q^2\left(\dfrac{a_{12}^2}{r^2}+\dfrac{a_{22}^2}{s^2}\right) &<& 1, \ \text{and} \\ e^{2\beta_2s_0}s^2\left((a_{11}p)^2+(a_{12}q)^2\right) &<& \left(pq(a_{11}a_{22}-a_{12}a_{21})\right)^2.
\end{eqnarray*}
 
\noindent If $A_1=J_1=\text{diag}(\alpha,\beta)$ and $A_2=J_2=\text{diag}(\gamma,\delta)$ are diagonal matrices, then $P=I$, $Q=I$, $P_2^{-1}P_1=\text{diag}(a,d)$ for some non-zero $a,d$. Hence inequalities~(\ref{eq:4}) are satisfied for some non-zero $a,d$ and $t,s>0$ if and only if  
\[
\max\{\vert a\vert e^{\alpha t}, \vert d\vert e^{\beta t}, e^{\gamma s}/\vert a\vert , e^{\delta s}/\vert d\vert \}<1.
\]
Note that this is satisfied provided either (i) $\alpha<0$, $\beta\geq 0$, $\gamma\geq 0$, and $\delta<0$, or (ii) $\alpha\geq 0$, $\beta< 0$, $\gamma< 0$, and $\delta\geq 0$. Let us assume (i) holds true (other case (ii) can be analyzed similarly). \\
If $\alpha<0\leq \beta$ and $\delta<0\leq \gamma$, it is easy to check that $a,d,t,s$ exist if and only if $\beta\gamma<\alpha\delta$ if and only if $A_1$ and $A_2$ have a Hurwitz convex combination. \\
Since $A_1$ and $A_2$ commute, this existence of a Hurwitz convex combination is a necessary and sufficient condition for stability, also see Example~\ref{exam:1}. \\
Observe that $D_1<1$ and $D_2<1$ implies $e^{2(\beta_2-\beta_1)s_0}<e^{2(-\alpha_2+\alpha_1)t_0}$, which is impossible for any positive $t_0,s_0$, if $\beta_2\geq \beta_1$ and $\alpha_2\geq\alpha_1$ (that is, if $\text{trace}(A_1)\geq 0$ and $\text{trace}(A_2)\geq 0$). Thus, we have the following proposition.

\begin{prop} \label{prop:trace}
For planar systems, if $s=(i_1,\dots,i_p,i_1)$ is a loop in $\G$ , then for the hypothesis of Theorem~\ref{thm:main} to satisfied, the trace of all the subsystem matrices $A_{i_1},\dots,A_{i_p}$, cannot be non-negative.  
\end{prop}
\begin{proof}
For the hypothesis of Theorem~\ref{thm:main} to satisfied, $\Vert P_{(i_j,i_{j+1})}e^{\eta_{(i_j,i_{j+1})}J_{i_j}}\Vert<1$, for all $j=1,\dots,p$, with $\eta_{(i_j,i_{j+1})}>0$, $i_{p+1}=i_1$. Since the switched system is planar, from the preceding discussion, $\text{det}(P_{(i_j,i_{j+1})}e^{\eta_{(i_j,i_{j+1})}J_{i_j}})<1$, for all $j=1,\dots,p$. Taking a product of all the terms, we get 
\[
1>\prod_{j=1}^p\text{det}(P_{(i_j,i_{j+1})}) \text{det}(e^{\eta_{(i_j,i_{j+1})}J_{i_j}})=\prod_{j=1}^p e^{\eta_{(i_j,i_{j+1})}\text{trace}{J_{i_j}}},
\] 
since $\prod_{j=1}^p P_{(i_j,i_{j+1})}=I$. Since all $\eta_{(i_j,i_{j+1})}>0$, the above inequality is not satisfied when the trace of all the subsystem matrices $A_{i_1},\dots,A_{i_p}$, are non-negative.  
\end{proof}
\end{exam}

\section{Examples}\label{sec:exam}
\begin{exam}\label{Xiang_Xiao}
The following system taken from~\cite{WW} for highlighting a comparison of our results with existing literature. Consider a planar switched system with the underlying graph $\G$ as a unidirectional ring with edge set $\E(\G)=\{(1,2), (2,1)\}$, and subsystem matrices
\[
A_1=\begin{pmatrix}
-1.9&0.6\\ 0.6&-0.1
\end{pmatrix}, \ \ A_2=\begin{pmatrix}
0.1 & -0.9 \\ 0.1 & -1.4
\end{pmatrix}.
\] 
The system is used by authors of~\cite{WW} to illustrate Theorem 2 in their paper, which gives sufficient conditions for stability of a switched system with all unstable subsystems. The sufficient conditions in~\cite{WW} involves solving a large number of matrix inequalities, which is calculation intensive. We show that this system satisfies the hypothesis of our main Theorem~\ref{thm:main}. Since this system is planar, our sufficient conditions for stability just reduce to solving the set of four inequalities given in~(\ref{eq:ineq}), involving $T_1, T_2, D_1, D_2$, as described earlier.

\begin{figure}[h!]
\centering
\begin{subfigure}{.5\textwidth}
  \centering
  \includegraphics[width=.9\linewidth]{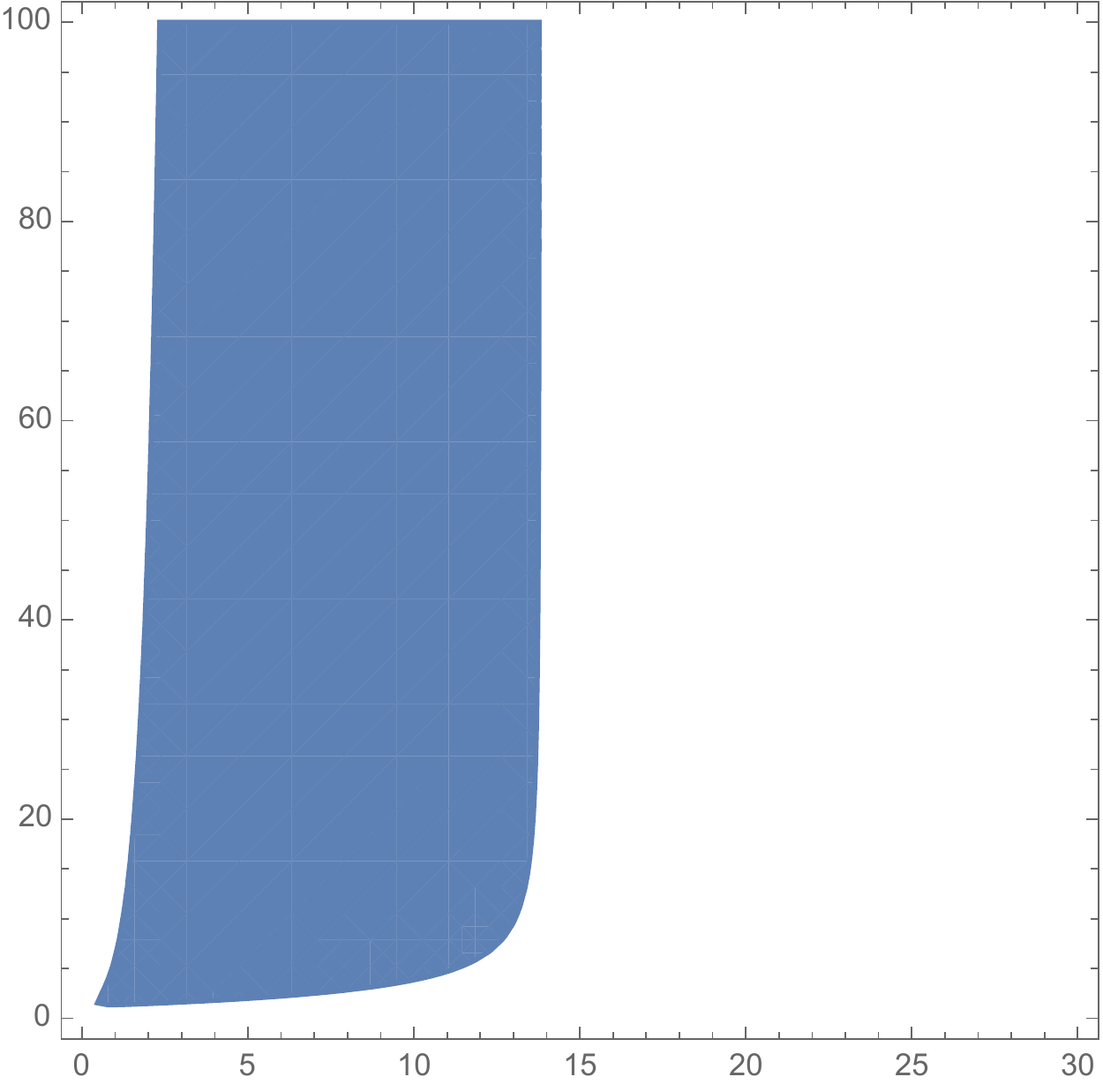}
  \caption*{(a)}
\end{subfigure}%
\begin{subfigure}{.5\textwidth}
  \centering
  \includegraphics[width=.9\linewidth]{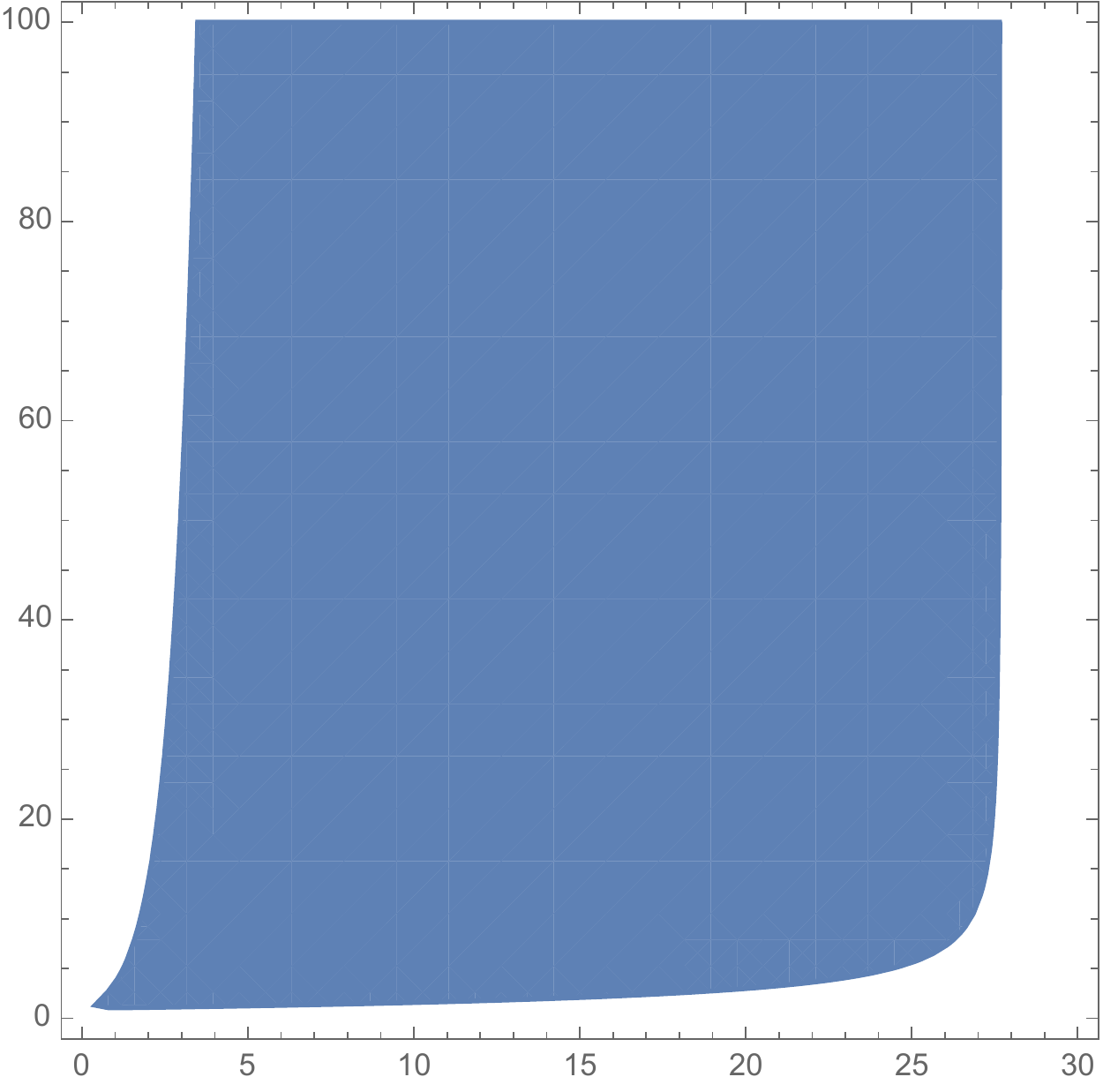}
    \caption*{(b)}
\end{subfigure}
\caption{Shaded region represents the values of $(t,x)$, where (a) $\Vert P_2^{-1}P_1e^{J_1t}\Vert < 1$, (b) $\Vert P_1^{-1}P_2e^{J_2t}\Vert < 1$.}
\label{xx_fig}
\end{figure}

\begin{figure}[h!]
\centering
\includegraphics[width=.5\textwidth]{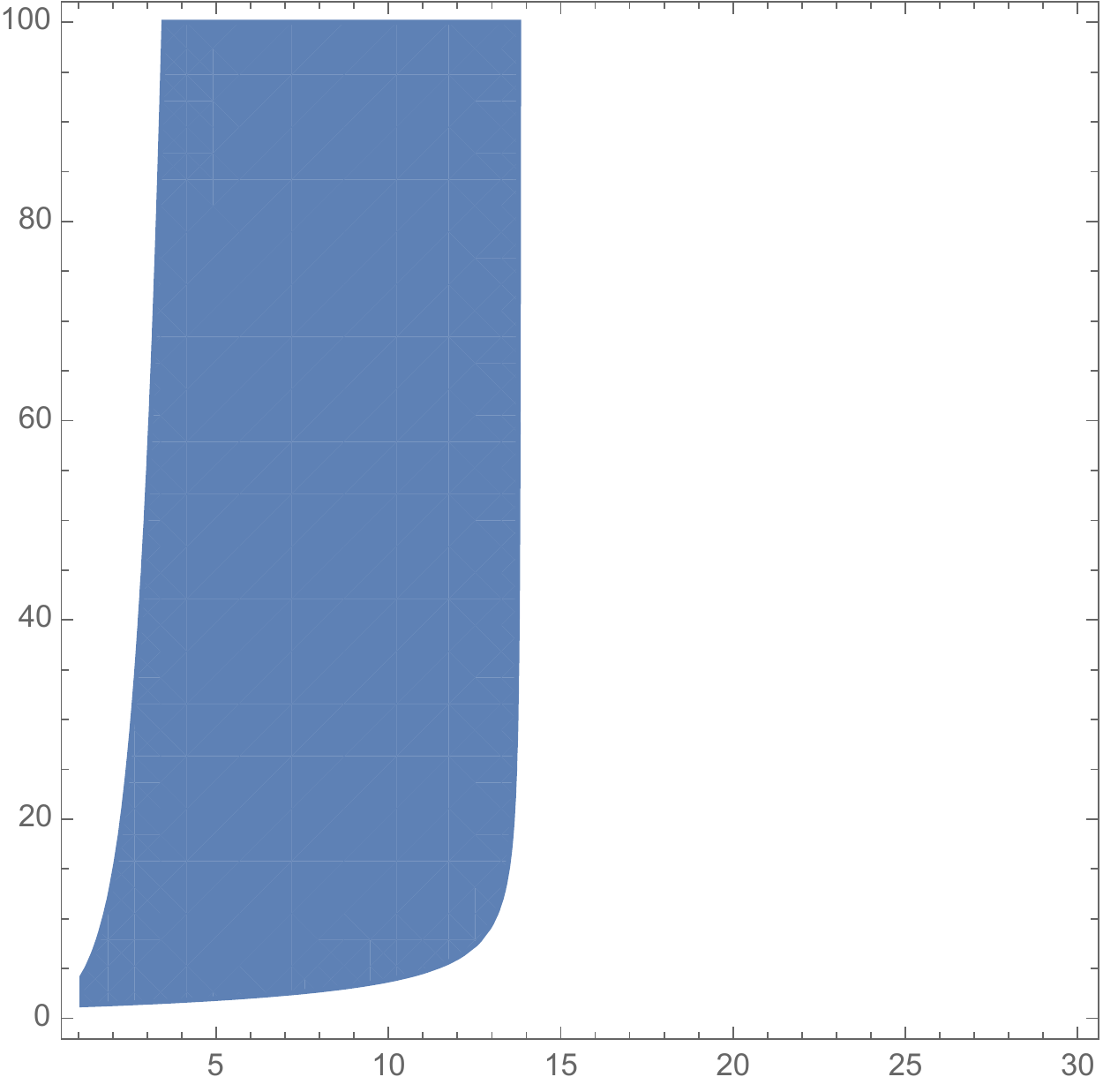}
\caption{Shaded region represents the values of $(t,x)$, where both $\Vert P_2^{-1}P_1e^{J_1t}\Vert < 1$ and $\Vert P_1^{-1}P_2e^{J_2t}\Vert < 1$.}
\label{xx_fig3}
\end{figure}

\noindent With the notation used in Example~\ref{exam:2}, we have 
\[
P=\begin{pmatrix}
-0.957092 & -0.289784 \\ 0.289784 &-0.957092
\end{pmatrix}, \ Q=\begin{pmatrix}
0.530691 & 0.997589 \\ 0.847565 & 0.069403
\end{pmatrix},
\]
\[
J_1=\begin{pmatrix}
-2.08167 & 0\\ 0 & 0.0816654
\end{pmatrix},\ J_2=\begin{pmatrix}
-1.33739 &0\\ 0 & 0.0373864
\end{pmatrix}.
\]
Assuming $D_1=D_2$ and setting $p/q=x$, equations~(\ref{eq:TD}) become
\begin{eqnarray}\label{eq:TD'}
T_1&=&e^{-2\alpha_1t_0}\left(a_{11}^2+a_{21}^2x^2\right)+e^{2\alpha_2t_0}\left(\dfrac{a_{12}^2}{x^2}+a_{22}^2\right),\\
D_1&=&e^{2(\alpha_2-\alpha_1)t_0}(a_{11}a_{22}-a_{12}a_{21})^2,\nonumber\\
T_2&=&\dfrac{e^{-2\beta_1s_0}\left(a_{21}^2x^2+a_{22}^2\right)+e^{2\beta_2s_0}\left(a_{11}^2+a_{12}^2x^2\right)}{(a_{11}a_{22}-a_{12}a_{21})^2},\nonumber\\
D_2&=&\dfrac{e^{2(\beta_2-\beta_1)s_0}}{(a_{11}a_{22}-a_{12}a_{21})^2}\nonumber.
\end{eqnarray}
Figures~\ref{xx_fig} and~\ref{xx_fig3} are plots in $(t,x)$-plane. For each value of $x$ (on the vertical axis), Figure~\ref{xx_fig}(a) shows the allowed values of $t$ (in the shaded region) for which $\Vert P_2^{-1}P_1e^{J_1t}\Vert < 1$, and Figure~\ref{xx_fig}(b) shows the allowed values of $t$ (in the shaded region) for which $\Vert P_1^{-1}P_2e^{J_2t}\Vert < 1$. Further, for each value of $x$ (on the vertical axis), Figure~\ref{xx_fig3} shows the allowed values of $t$ (in the shaded region) for which $\Vert P_2^{-1}P_1e^{J_1t}\Vert < 1$ and $\Vert P_1^{-1}P_2e^{J_2t}\Vert < 1$. From this, it is clear that the switched system is stable for all periodic signals $\sigma$ with $t_{n+1}-t_n=\tau$, for all $n\geq 0$, with any period $\tau$ between 2 and 13.
\end{exam}

\begin{exam}
\begin{figure}[h!]
\centering
\includegraphics[width=.7\textwidth]{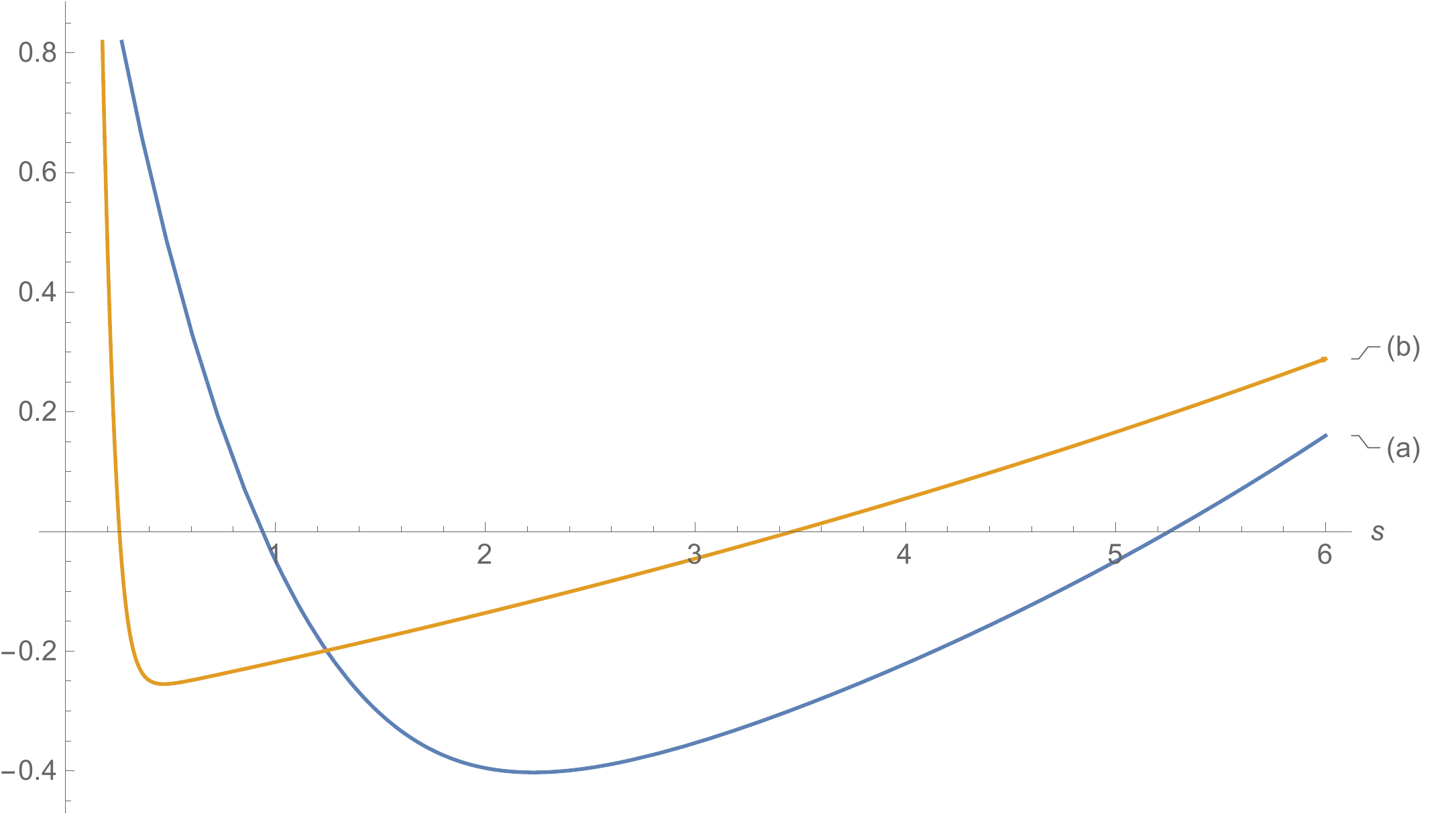}
\caption{(a) Plot of $\Vert P_{(1,2)}e^{J_1s}\Vert-1$ and (b) Plot of $\Vert P_{(2,1)}e^{J_2s}\Vert-1$.}
\label{exam1-fig4}
\end{figure}

%\begin{figure}[h!]
%\centering
%\includegraphics[width=.6\textwidth]{exam1-fig3.pdf}
%\caption{Plot of the real part of the eigenvalue $sA_1+(1-s)A_2$ with largest real part.}
%\label{exam1-fig3}
%\end{figure}

\begin{figure}[h!]
\centering
\includegraphics[width=.7\textwidth]{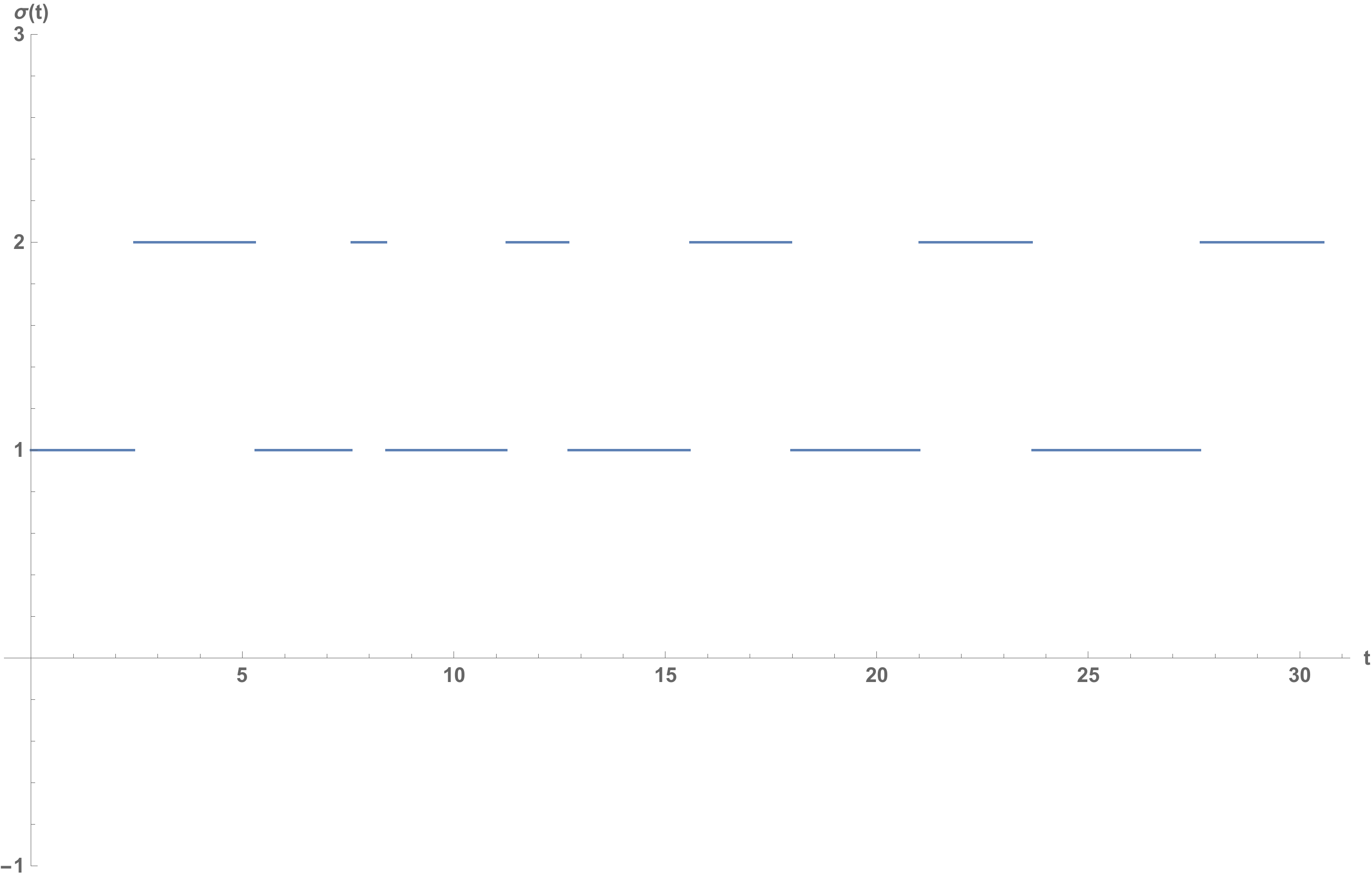}
\caption{Switching signal $\sigma$.}
\label{exam1-fig1}
\end{figure}

\begin{figure}[h!]
\centering
\includegraphics[width=.7\textwidth]{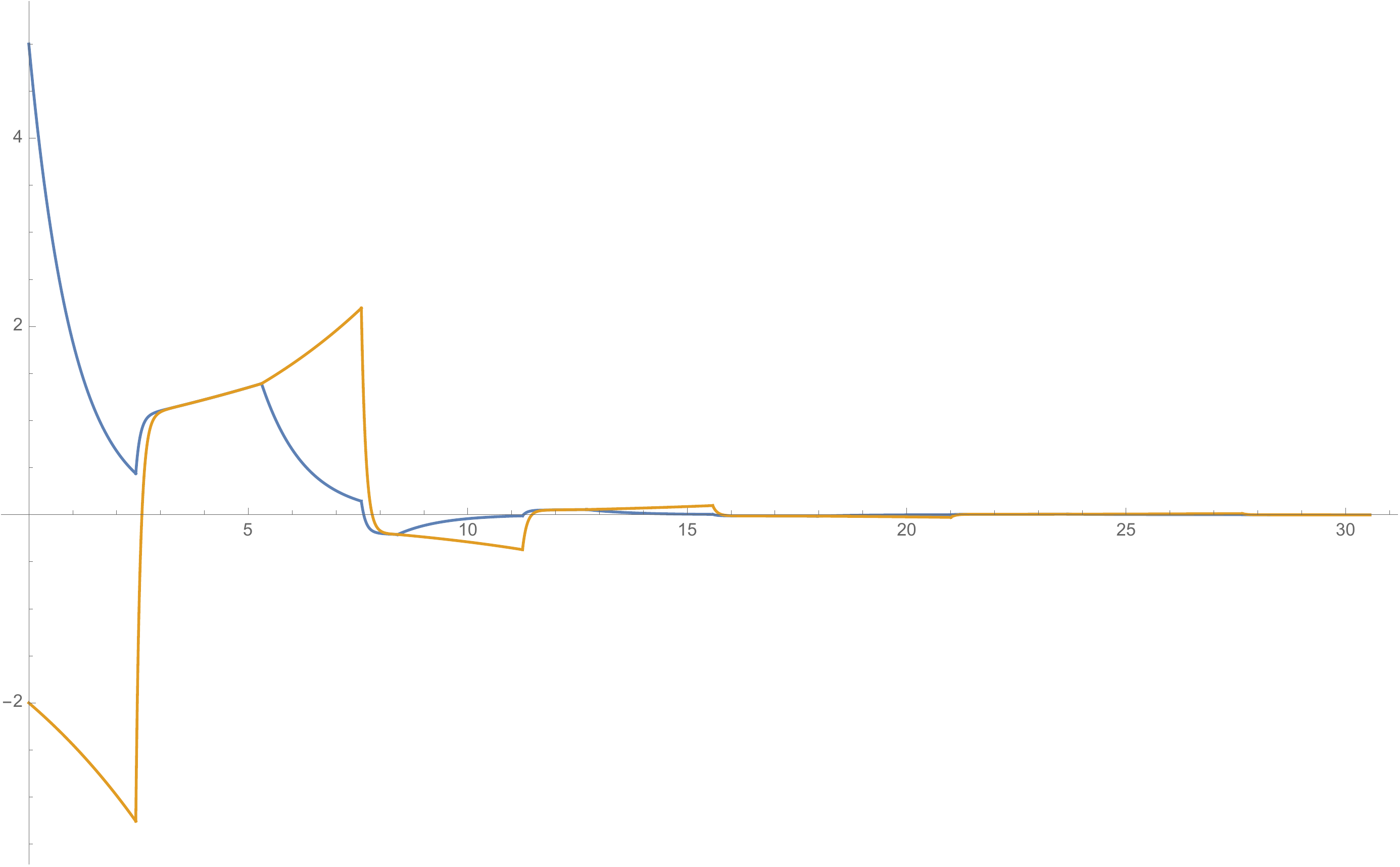}
\caption{Solution trajectories with initial condition $(5,-2)$ with switching signal $\sigma$.}
\label{exam1-fig2}
\end{figure}
Consider a planar switched system with underlying graph $\G$ as a unidirectional ring with edge set $\E(\G)=\{(1,2), (2,1)\}$, with subsystem matrices
\[
A_1=\begin{pmatrix}
-1&0\\ 0&0.2
\end{pmatrix},\ \ A_2=\begin{pmatrix}
1.76363 & -1.66363 \\ 11.7636& -11.6636
\end{pmatrix}. 
\]
Let $P_1=I$, $P_2=\begin{pmatrix}
\sqrt{2}&0.5\\ 10&0.5
\end{pmatrix}$, $J_2=\begin{pmatrix}
-10&0\\ 0&0.1
\end{pmatrix}$, $J_1=\begin{pmatrix}
-1&0\\ 0&0.2
\end{pmatrix}$. \\
For $s\in I_{(2,1)}=(0.5,3)$ and $t\in I_{(1,2)}=(1,4)$, we get $\Vert P_{(2,1)}e^{J_2s}\Vert <1$ and $\Vert P_{(1,2)}e^{J_1t}\Vert<1$, see Figure~\ref{exam1-fig4}. Note that the matrices $A_1$ and $A_2$ are non-commuting and there exists a Hurwitz convex combination of these matrices. Consider a switching signal $\sigma$, shown in Figure~\ref{exam1-fig1}, with randomly chosen first twelve switching times within the allowed interval range $I_{(1,2)}$ and $I_{(2,1)}$,
\begin{eqnarray*}
(2.43717, 2.86591, 2.27316, 0.826817, 2.84621, 1.46092, \\
2.87292,
2.39123, 3.033, 2.66629, 3.98035, 2.90419).
\end{eqnarray*}
Figure~\ref{exam1-fig2} shows the convergence of solution trajectory of the switched system with this switching signal $\sigma$.
\end{exam}

\begin{exam}
In these examples, hypotheses of Theorem~\ref{thm:main} are not satisfied if we take $P_i$ to have unit norm columns. Moreover inequalities~(\ref{eq:5}) are not satisfied for any choice of diagonal matrices.\\
a) Consider a planar switched system with underlying graph $\G$ as a unidirectional ring with edge set $\E(\G)=\{(1,2), (2,1)\}$, and subsystem matrices 
\[
A_1=\begin{pmatrix}
1&1\\ 3&0.4
\end{pmatrix}, \ A_2=\begin{pmatrix}
2&1\\ 0.1&-0.6
\end{pmatrix}.
\]
%Here 
%\[
%J_1=\begin{pmatrix}
%2.45784&0\\ 0&-1.05784
%\end{pmatrix}, \ J_2=\begin{pmatrix}
%-2.03791& 0 \\ 0&0.637909
%\end{pmatrix}.
%\]
%%If we insist on columns of $P_i$ having unit norm, then 
%%\[
%%P_1=\begin{pmatrix}
%%0.565659& -0.437073 \\ 0.82464 & 0.899426
%%\end{pmatrix}, \ P_2=\begin{pmatrix}
%%-0.999282 & 0.354473 \\ -0.0378816 & -0.935066
%%\end{pmatrix}.
%%\]
Both $A_1$ and $A_2$ are non-commuting unstable matrices with positive trace. Moreover, no convex combination of these matrices is Hurwitz. \\
b) Consider a planar switched system with underlying graph $\G$ with edge set $\E(\G)=\{(1,2), (1,4), (2,3), (3,1), (4,1)\}$, and subsystem matrices 
\[
A_1=\begin{pmatrix}
1&-1\\ 1&1
\end{pmatrix}, \ A_2=A_3=\begin{pmatrix}
2&1\\0&-3
\end{pmatrix}, \ A_4=\begin{pmatrix}
4&-1\\-1&-3
\end{pmatrix}.
\]
All of the subsystem matrices are unstable and the matrices $A_1$ and $A_4$ have positive trace. Hence, by Proposition~\ref{prop:trace}, the switched system does not satisfy the hypothesis of Thereom~\ref{thm:main}.  
\end{exam}

\begin{exam}
\begin{figure}[h!]
\centering
\includegraphics[width=.8\textwidth]{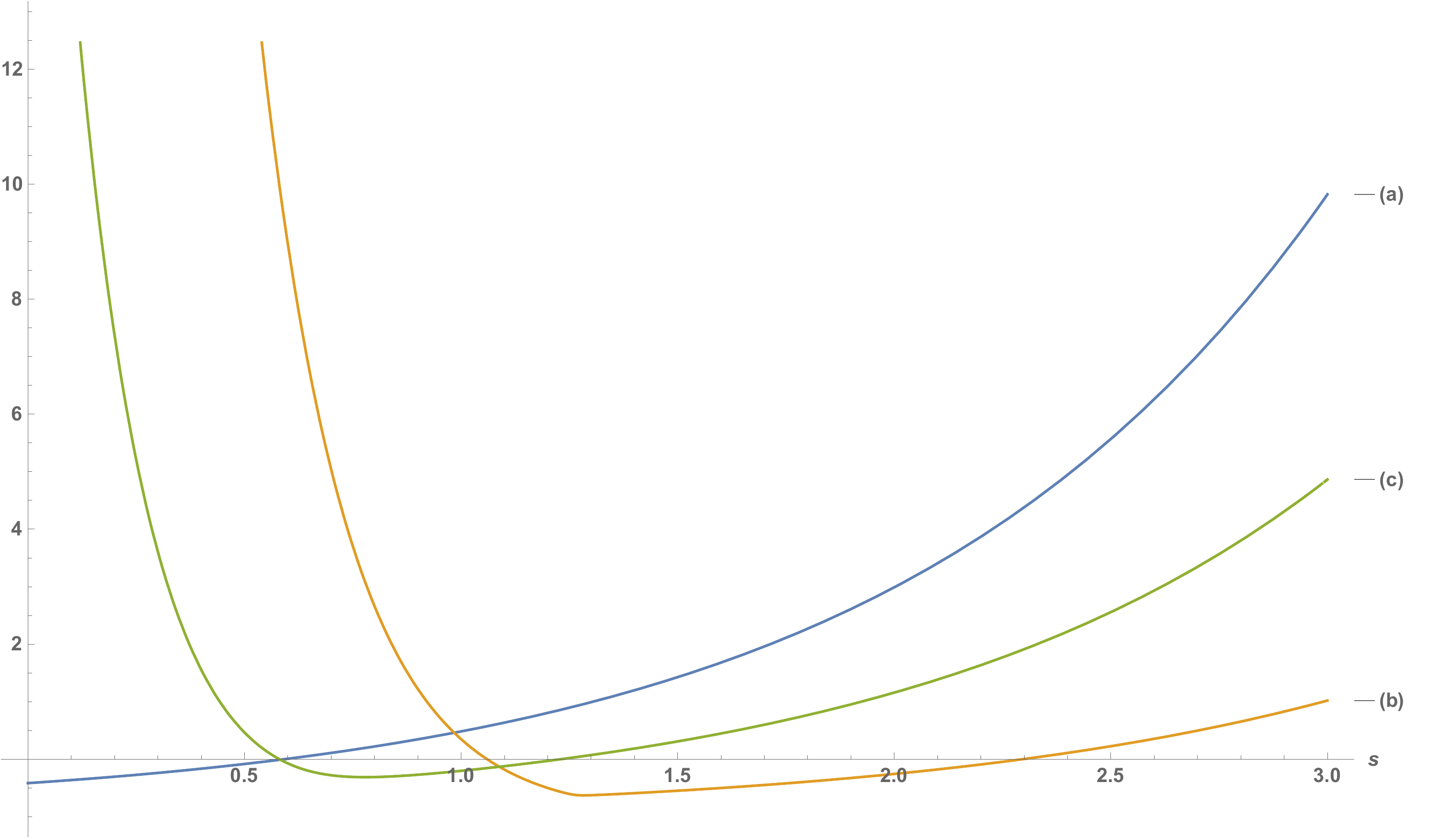}
\caption{(a) Plot of $\Vert P_{(1,2)} e^{J_1s}\Vert-1$, (b) Plot of $\Vert P_{(2,3)} e^{J_2s}\Vert-1$, and (c) Plot of $\Vert P_{(3,1)} e^{J_3s}\Vert-1$.}
\label{exam2-fig1}
\end{figure}
Consider a unidirectional ring $\G$ with three vertices and let
\[
A_1=\begin{pmatrix}
1&0\\ 0.9&0.1
\end{pmatrix}, \ A_2=\begin{pmatrix}
0.538462 &1.38462 \\ 1.84615 & -4.53846
\end{pmatrix}, \ A_3=\begin{pmatrix}
26.8725 & -98.6387\\ 8.62228& -31.8725
\end{pmatrix}.
\]
Here $J_1=\text{diag}(1,0.1)$, $J_2=\text{diag}(-5,1)$, $J_3=\text{diag}(1,-6)$.\\
See Figure~\ref{exam2-fig1}, hypothesis of Theorem~\ref{thm:main} is satisfied for 
\[
P_1=\begin{pmatrix}
1 &  0 \\ 1& 1
\end{pmatrix}, \ P_2=\begin{pmatrix}
-0.769231 & 2.30769 \\ 3.07692 & 0.769231
\end{pmatrix}, \ P_3=\begin{pmatrix}
-0.23485 & 23.1004 \\ -0.0616001 & 7.69847
\end{pmatrix}.
\] 
Here $\Vert P_{(1,2)}\Vert<1$, $\Vert P_{(2,3)}\Vert>1$ and $\Vert P_{(3,1)}\Vert>1$.

\end{exam}

\section{Concluding Remarks}\label{conclusion}
In this paper, we have given sufficient stability conditions for switched systems, which are particularly useful for switched systems with all non-Hurwitz subsystems. Several examples are given to illustrate the applicability of our result. Even though it is easy to check when the hypothesis of Theorem~\ref{thm:main} are not valid using Remark~\ref{rem:hyp}(5) and Proposition~\ref{prop:trace}, it is not straightforward to find sufficient conditions only in terms of the subsystem matrices $A_1,\dots,A_k$ and the architecture of the underlying graph $\G$, under which the hypotheses hold true. For planar systems, hypotheses of Theorem~\ref{thm:main} can be reduced to simple computable conditions as discussed in Example~\ref{exam:2} and also in Proposition~\ref{prop:trace}. Example~\ref{Xiang_Xiao} provides a comparison of our result with the existing result in the literature. An analytical comparison of the sufficient conditions presented here with the conditions available in the literature is an ongoing project. Further applicability of our results to large scale systems and estimating computation costs can be explored.

\section{Funding}
This work was funded by Science Engineering Research Board, Department of Science and Technology, India (File No. YSS/2014/000732).

\bibliographystyle{plain}
\bibliography{mybibfile}

\end{document}